\documentclass{amsart}
\usepackage[latin1]{inputenc}
\usepackage{subfigure}
\usepackage{amsmath}
\usepackage{graphicx}
\newcommand{\Ne}{\operatorname{ne}} 
\newcommand{\Cr}{\operatorname{cr}} 
\newcommand{\NE}{\operatorname{N_n}} 
\newcommand{\CR}{\operatorname{C_n}} 
\newcommand{\Sn}{\mathfrak{S}_n}
\newcommand{\vtype}[1]{
\begin{minipage}[c]{0.6cm}\center\mbox{}\smallskip\\\includegraphics[scale=0.5]{#1.png}\\\mbox{}\end{minipage}}
\newtheorem{theorem}{Theorem}

\graphicspath{{Images/}}

\usepackage[round]{natbib}
\title{On $k$-crossings and $k$-nestings of permutations}
\author{Sophie Burrill%
     \and Marni Mishna%
     \and {Jacob Post}%
}
\address{%
 [SB, MM] Department of Mathematics, Simon Fraser University, Burnaby Canada, V5A 1S6;
 [JP] Department of Computer Science, Simon Fraser University, Burnaby Canada, V5A 1S6}

\keywords{crossing, nesting, permutation, enumeration}

\begin{document}
\begin{abstract}
We introduce $k$-crossings and $k$-nestings of permutations. We show that the crossing number and the nesting number of permutations have a symmetric joint distribution.  As a corollary, the number of $k$-noncrossing permutations is equal to the number of $k$-nonnesting permutations. We also provide some enumerative results for $k$-noncrossing permutations for some values of $k$. 
\end{abstract}
\maketitle

\section{Introduction}
Nestings and crossings are equidistributed in many combinatorial objects, such as matchings, set partitions, permutations, and embedded labelled graphs~\cite{Chetal07, Corteel07, deMi07}. More surprising is the symmetric joint distribution of the crossing and nesting numbers:  A set of $k$~arcs forms a $k$-crossing (nesting) if each of the $\binom{k}{2}$ pairs of  arcs cross (nest).  The crossing number of an object is the largest $k$ for which there is a $k$-crossing, and the nesting number is defined similarly.  Chen \emph{ et al.}~ \cite{Chetal07} proved the symmetric joint distribution of the nesting and crossing  numbers for set partitions and matchings. Although they describe explicit involutions, they do not use simple local operations on the partitions. De Mier~\cite{deMi07}, interpreted the work of Krattenthaler~\cite{Kratt06} to show that $k$-crossings and $k$-nestings satisfy a similar distribution in embedded labelled graphs. 

A hole in this family of results is the extension of $k$-crossings and $k$-nestings in permutations. This note fills this gap. We also give exact enumerative formulas for permutations of size $n$ with crossing numbers $1$ (non-crossing) and $\lceil n/2\rceil$.

\section{Introducing $k$-crossings and $k$-nestings of permutations}
\subsection{Crossings and nestings}
The \emph{ arc annotated sequence}  associated to the permutation $\sigma\in\Sn$ is the directed graph on the vertex set $V(\sigma)=\{1,\dots, n\}$ with arc set $A(\sigma)=\{(a, \sigma(a)): 1\leq a\leq n\}$, drawn in a particular way. It is also known as the standard representation, or simply, the arc diagram.  It is embedded in the plane by drawing an increasing linear sequence of the vertices, with edges $(a, \sigma(a))$ satisfying  $a\leq \sigma(a)$ drawn above the vertices (the upper arcs), and the remaining lower arcs satisfying $a>\sigma(a)$ drawn below. We refer to this graph as $A_\sigma$; the subgraph induced by the upper arcs and  $V(\sigma)$ is $A_\sigma^{+}$; and the subgraph induced by the lower arcs and $V(\sigma)$ is $A_\sigma^{-}$. Additionally, we reverse the orientation of the arcs in $A_\sigma^{-}$, and view it as a classic arc diagram above the horizon.  Because of these rules, the direction of the arcs is determined, and hence we simplify our drawings by not showing arrows on the arcs.  

These two subgraphs are arc diagrams in their own right: for example $A_\sigma^{-}$ represents a set partition, and $A_\sigma^+$ is a set partition with some additional loops. For convenience purposes, we draw $A_{\sigma}^{-}$ flipped above the axis, so that it resembles other arc diagrams.

\begin{figure}\center
\large{$A_\sigma=$} \includegraphics[width=6cm]{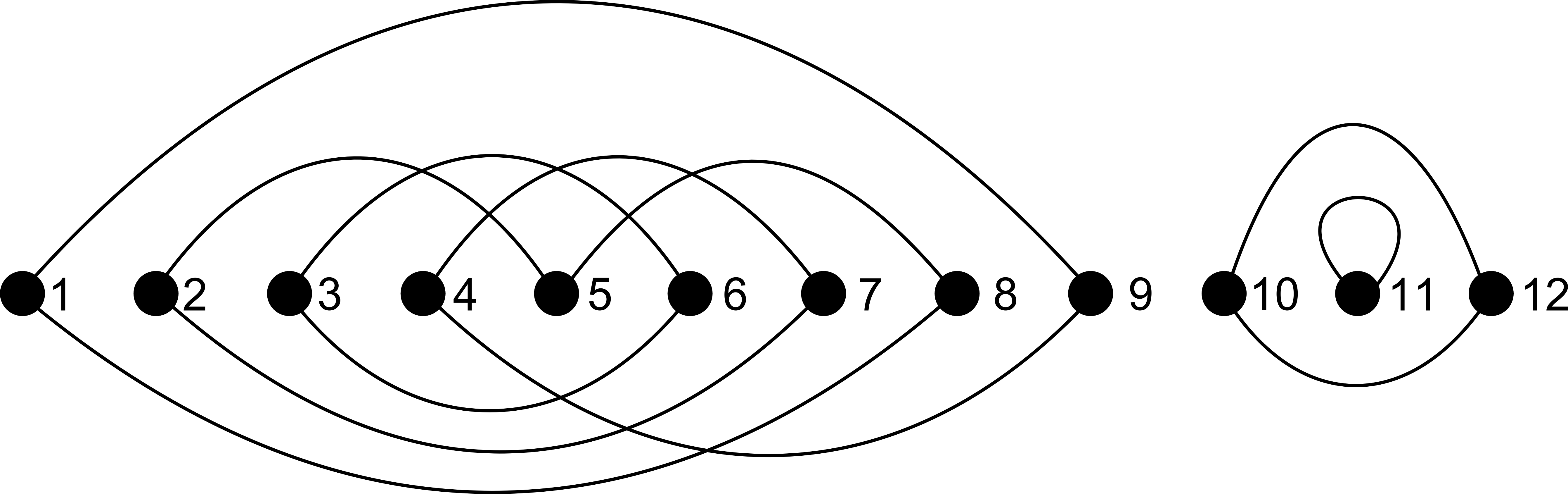} \\ \bigskip
\large{$A_{\sigma^{+}}=$} \includegraphics[width=6cm]{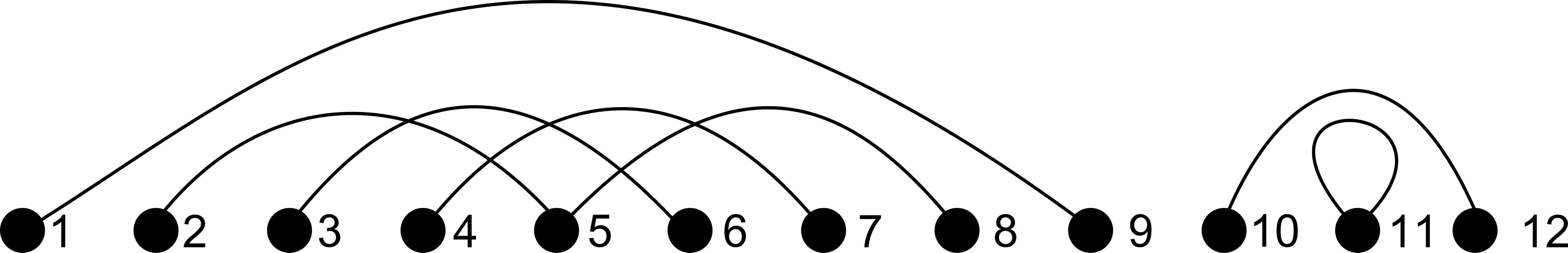}\\ \bigskip
\large{$A_{\sigma^{-}}=$}  \includegraphics[width=6cm]{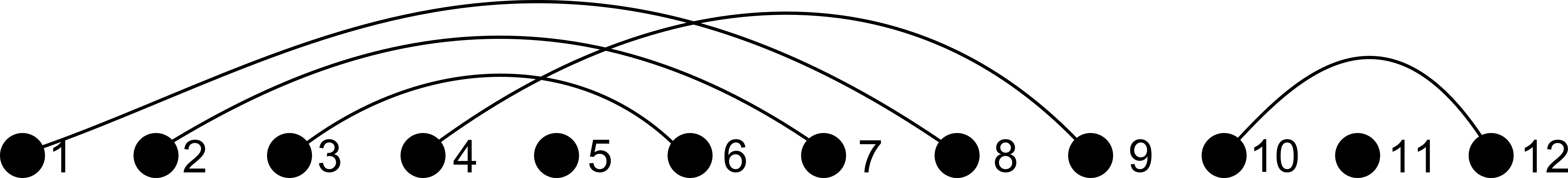}\\ \bigskip

\caption{An arc diagram representation for the permutation \mbox{$\sigma=[9\,5\,6\,7\,8\,3\,2\,1\,4\,12\,11\,10]$}, and its decomposition into upper and lower arc diagrams $(A_\sigma^{+}, A_\sigma^{-})$. In this example,  $\Cr(\sigma)=4$, $\Ne(\sigma)=3$, and the degree sequence is given by $D_{\sigma}={ (1,0)(1,0)(1,0)(1,0)(1,1)(0,1)(0,1)(0,1)(0,1)(1,0)(1,1)(0,1)}.$}
\end{figure}

Crossings and nestings are defined for permutations by considering the upper and lower arcs separately. A crossing is a pair of arcs $\left\{(a,\sigma(a)),(b, \sigma(b))\right\}$ satisfying either $a<b\leq \sigma(a)<\sigma(b)$ (an upper crossing) or $\sigma(a)<\sigma(b)<a<b$ (a lower crossing). A nesting is a pair of arcs $(a,\sigma(a))$ $(b, \sigma(b))$ satisfying $a<b\leq \sigma(b)< \sigma(a)$ (an upper nesting) or $\sigma(a)<\sigma(b)<b<a$ (a lower nesting).

There is a slight dissymmetry to the treatment of upper and lower arcs in this definition which we shall see is inconsequential. However, the reader should recall that what is considered a crossing (nesting) in the upper diagram is elsewhere called an \emph{enhanced} crossing (resp. enhanced nesting). 

Crossings and nestings were defined in this way by Corteel~\cite{Corteel07} because they represent better known permutation statistics. Corteel's Theorem 1 states that the number of top arcs in this representation of a permutation is equal to the number of weak exceedences, the number of arcs on the bottom is the number of descents, each crossing is equivalent to an occurrence of the pattern $2-31$, and each nesting is an occurrence of the pattern $31-2$. Corteel's Proposition 4 states nestings and crossings occur in equal number across all permutations of length $n$.

\subsection{$k$-nestings and $k$-crossings}
To generalize her work we define $k$-crossings and $k$-nestings in the same spirit as set partitions and matchings. A $k$-crossing in a permutation ard diagram $A_\sigma$ is a set of~$k$ arcs~$\{(a_i, \sigma(a_i)):1\leq i\leq k\}$ that satisfy either the relation $a_1<a_2<\dots<a_k<\sigma(a_1)<\sigma(a_2)<\dots<\sigma(a_k)$ (upper $k$-crossing)
or $\sigma(a_1)<\sigma(a_2)<\dots<\sigma(a_k)<a_1<a_2<\dots<a_k$ (lower $k$-crossing). Similarly, a $k$-nesting is a set of $k$ arcs $\{(a_i, \sigma(a_i)):1\leq i\leq n\}$ that satisfy either the relation $a_1<a_2<\dots<a_k<\sigma(a_k)<\dots<\sigma(a_2)<\sigma(a_k)$ (upper $k$-nesting)
or $\sigma(a_1)<\sigma(a_2)<\dots<\sigma(a_k)<a_k<\dots<a_2<a_1$ (lower $k$-nesting). 

The \emph{ crossing number\/} of a permutation $\sigma$, denoted by $\Cr(\sigma)$, is the size of the largest~$k$ such that $A_\sigma$ contains a $k$-crossing. In this case we also say $\sigma$ is $k+1$-noncrossing. Likewise, the nesting number of a permutation $\Ne(\sigma)$ is the size of the largest nesting in $A_\sigma$, and define $k+1$-noncrossing similarly. Occasionally we consider the top and lower diagrams in their own right as \emph{ graphs}, and then we use the definition of deMier~\cite{deMi07}, and hence distinguish separately the enhanced crossing number of the \emph{ graph} $A_\sigma^+$ denoted $\Cr^*(A_\sigma^+)$ from the permutation crossing number, and likewise for the enhanced nesting number $\Ne^*$.  The number of permutations of $\Sn$ with crossing number equal to $k$ is $\CR(k)$, and we likewise define $\NE(k)$ for nestings.

The \emph{ degree sequence} $D_g $ of a graph $g$ is the sequence of indegree and outdegrees of the vertices, when considered as a directed graph:
\[D_g \equiv(D_g(i))_i=\left(\operatorname{indegree}_{g}(i), \operatorname{outdegree}_{g}(i)\right)_{i=1}^n.\] Some sources call these left-right degree sequences since in other arc diagrams the incoming arcs always come in on the left, and the outgoing arcs go out to the right. 
As a graph, the degree sequence of a permutation is trivial: $(1,1)^n$, since a permutation is a map in which every point has a unique image, and a unique pre-image. To define a more useful entity, we define the degree sequence of a permutation to be the degree sequence of only the upper arc diagram: $D_\sigma\equiv D_{A_\sigma^+}$. The degree sequence defined by the lower arc diagram can be computed coordinate-wise directly from the upper by simple transformations given in Table~\ref{tab:vtypes}, and we denote this sequence $\overline{D_\sigma}$. (The sums of the vertex degrees is not (1,1) because the lower arcs have their orientation reversed, and hence the indegree, and the out degree have switched)
An example is in Figure 1. The vertices with degree $(0,1)$  are called ``openers'' and those with degree $(1,0)$ are ``closers''. 

\begin{table}\label{tab:vtypes}
\center
\begin{tabular}{|l|c|c|l|}\hline
{\bf Type} & {\bf vertex $i$} & $D_{\sigma}(i)$ &$\overline{D_{\sigma}}(i)$ \\\hline
opener & \vtype{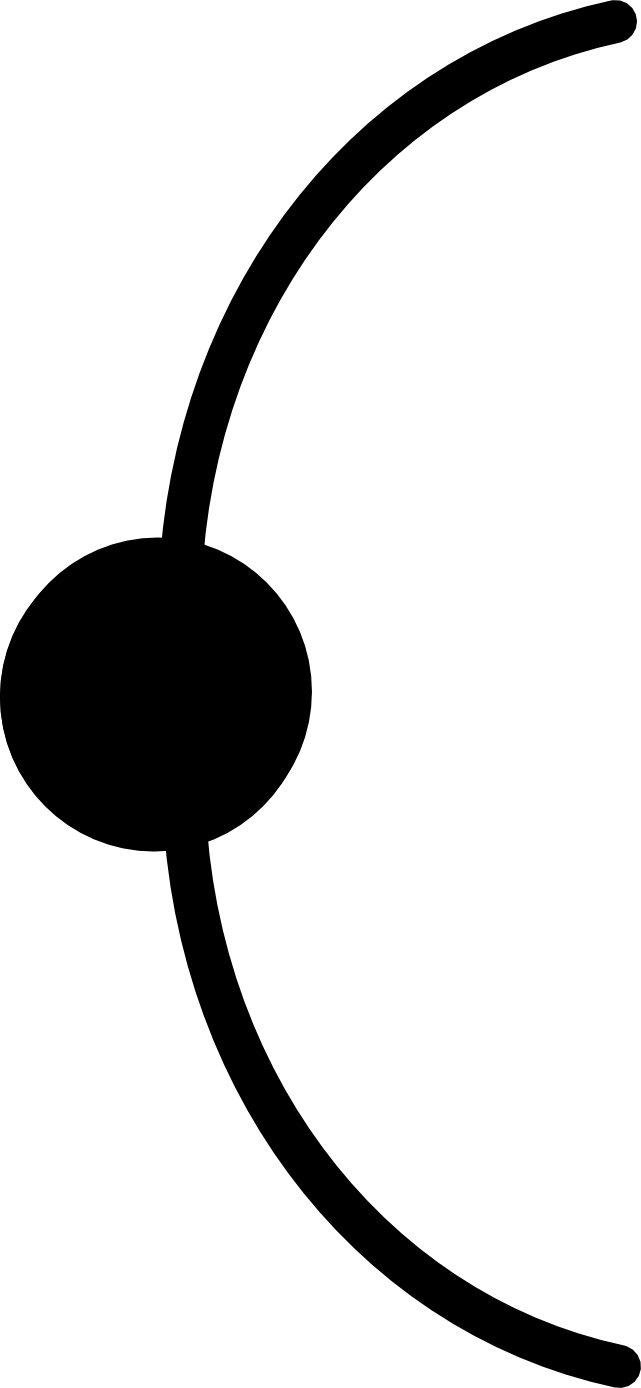} &(1,0) & (1,0)\\ \hline
closer   & \vtype{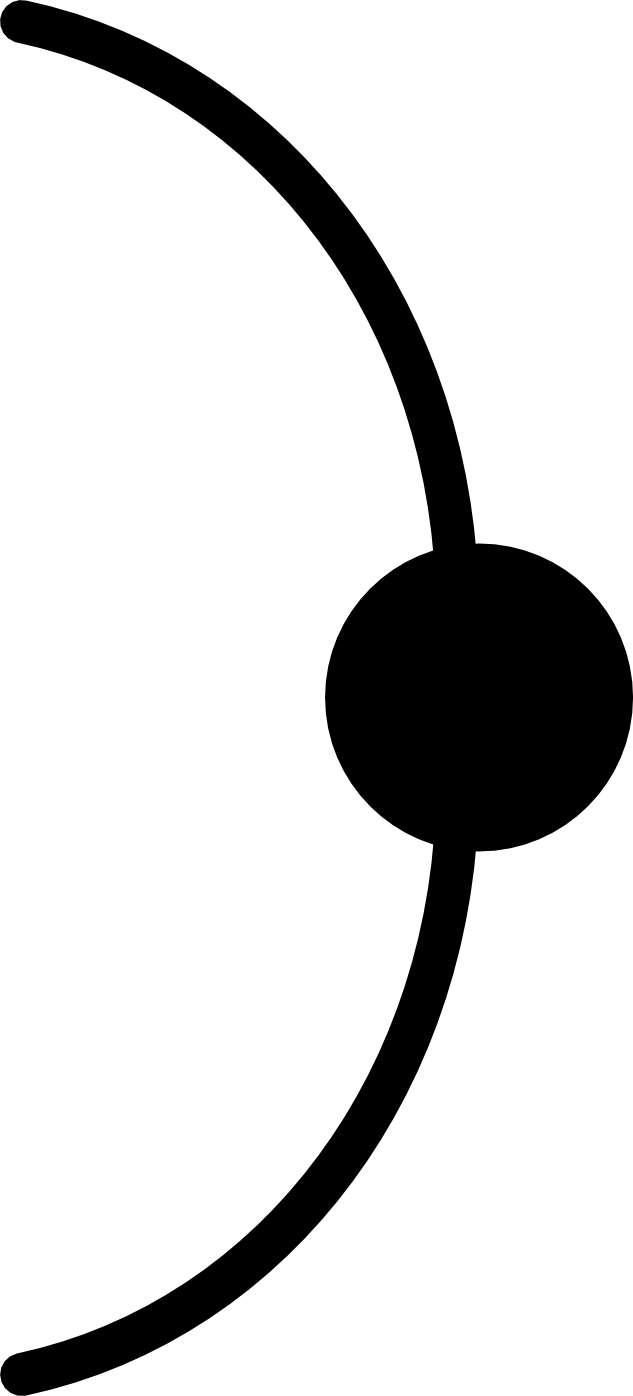}  &(0,1) & (0,1)\\ \hline
loop  & \vtype{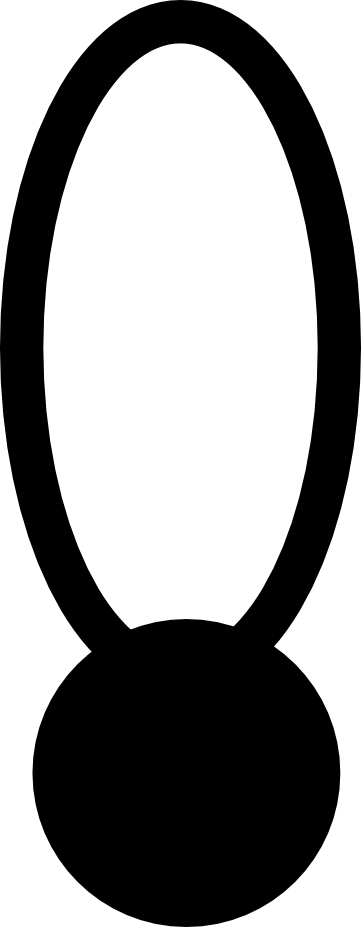} & (1,1) &  (0,0)\\\hline
 upper transient & \vtype{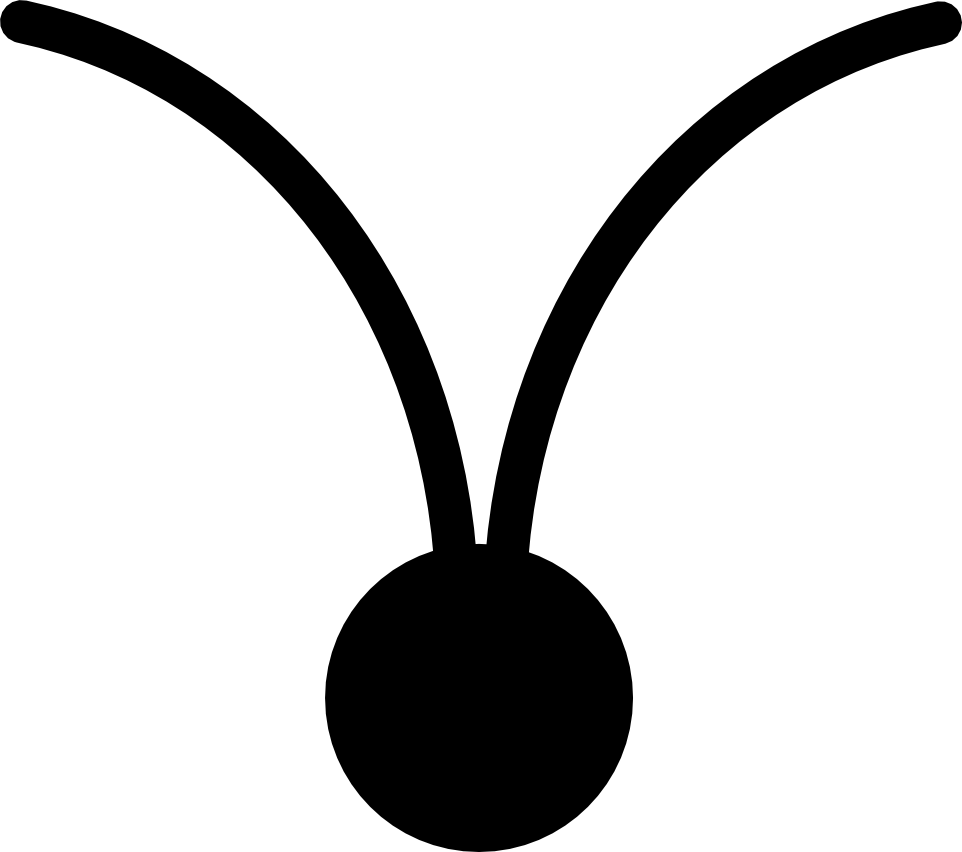} &(1,1) &  (0,0)\\ \hline
lower transient & \vtype{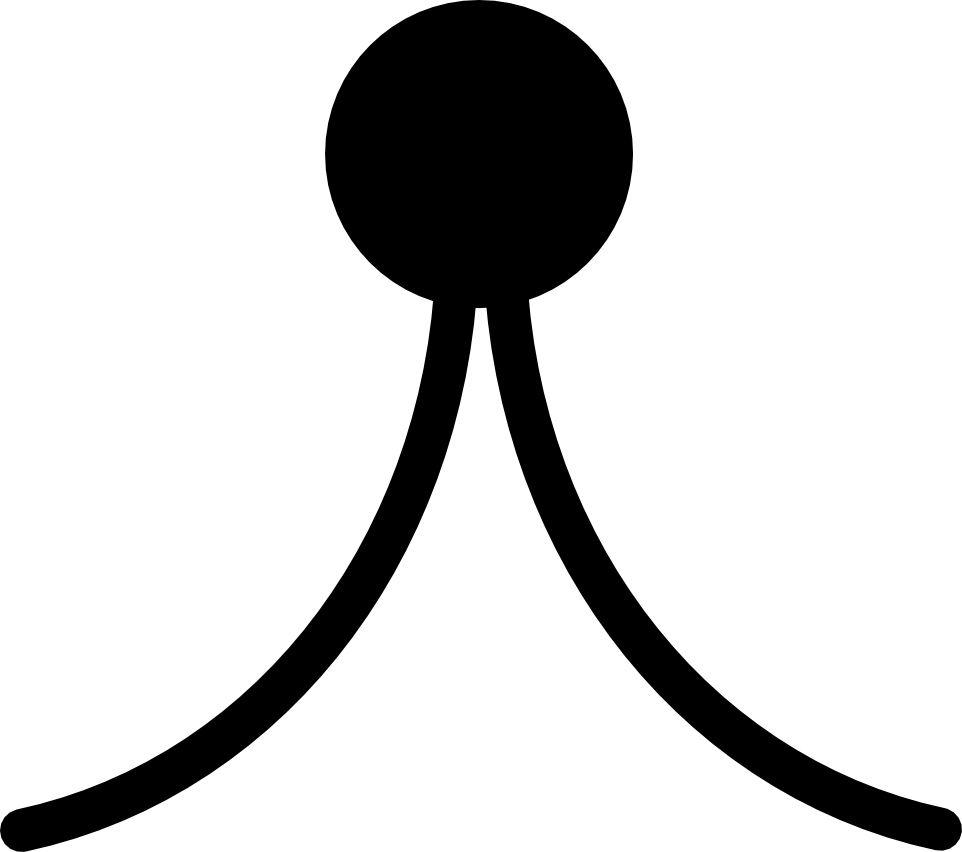} &(0,0) &  (1,1)\\ \hline
\end{tabular}

\smallskip
\caption{\it The five vertex types that appear in permutations, and their associated upper degree value, and lower degree value. }

\end{table}
The main theorem can now be stated. 

\begin{theorem} \label{thm:main}
Let $NC_n(i,j, D)$ be the number of permutations of~$n$ with crossing number $i$, nesting number $j$, and left-right degree sequence specified by $D$. Then 
\begin{equation}
NC_n(i,j, D)=NC_n(j,i, D).
\end{equation}
\end{theorem}
There is an explicit involution behind this enumerative result. 
\subsection{Preliminary enumerative results}
The number of permutations of $\Sn$ with crossing number equal to $k$ is directly computable for small values of~$n$ and~$k$. 
\begin{table}[h!]\center
\begin{tabular}{c|rrrrrrrrr}
$n\backslash k$ & 1 & 2 & 3 & 4 & 5\\ \hline
1&1\\
2& 2 &\\
3& 5& 1\\
4 & 14 & 10 &   \\
5 & 42 & 76&2 \\
6 & 132& 543& 45 \\
7 & 429& 3904& 701& 6 \\
8 & 1430& 29034&9623& 233\\
9 & 4862& 225753& 126327& 5914& 24\\
\end{tabular}
\smallskip
\caption{\it $\CR(k)$: The number of permutations of $\Sn$ with crossing number $k$. A crossing number of 1 is equivalent to non-crossing. }
\label{tab:enum}
\end{table}

We immediately notice the first column of Table~\ref{tab:enum}, the non-crossing permutations,  are counted by Catalan numbers: $\CR(1)=\frac{1}{n+1}\binom{2n}{n}$. This has a simple explanation: non-crossing \emph{ partitions\/} have long been known to be counted by Catalan numbers and there is a simple bijection between non-crossing permutations and non-crossing partitions. 
Essentially, to go from a non-crossing permutation to a non-crossing partition, flip the arc diagram upside down, convert the loops to fixed points,  and then remove the lower arcs. This defines a unique set  partition, and is easy to reverse. This bijection is easy to formalize, but it is not the main topic of this note. 
\section{Enumeration of maximum nestings and crossings }
To get a sense of how Theorem 1 is proved, and to obtain some new enumerative results, we consider the set of maximum nestings and crossings. A \emph{maximum nesting} is the largest possible: a $\lceil n/2\rceil$-nesting is maximum in a permutation on~$n$ elements. We can compute $\NE( \lceil n/2\rceil )$ explicitly. 

\begin{theorem}
The number of permutations with a maximum nesting satisfies the following formula:
\begin{equation}
  \NE( \lceil n/2\rceil )= \left\{  \begin{array}{ll} 
                                                    m! & n=2m+1\\
                                                    2(m+1)!-(m-1)!-1 & n=2m\\
                                                    \end{array}.
                 \right.
\end{equation}
\end{theorem}
\begin{proof}
We divide the result into a few cases, but each one is resolved the same way: For each permutation $\sigma\in\Sn$ with a maximum nesting, the $ \lceil n/2\rceil $-nesting comes from either $A_\sigma^+$ or $A_\sigma^{-}$, and in most cases defines that subgraph. Once one side is fixed, and there is a given degree sequence, it is straightforward to compute the number of ways to place the remaining arcs. Some cases are over counted, and tallying these gives the final result.

\subsubsection*{Odd $n$: $n=2m+1$} To achieve an $m+1$-nesting, it must be an enhanced nesting in the upper arc diagram, and it uses all vertices, including a loop: $\sigma(i)=n-i: 1\leq i\leq m$. It remains to define $\sigma(i)$ for $m<i\leq n$. The lower degree sequence is fixed, and so $1\leq \sigma(i)<m$ for each~$i$, but other than that there is no restriction. Thus, there are $m!$ possibilities.

\subsubsection*{Even $n$: $n=2m$} The even case is slightly more complicated, owing to the fact that three different ways to achieve an $m$-nesting:
\begin{description}
\item[An $m$-nesting in $A_\sigma^+$]
These permutations satisfy $\sigma(i)=n-i, 1\leq i
\leq m$. As before, there are $m!$ ways to define $\sigma(i), m<i\leq n$. 

\item[An $m$-nesting in $A_\sigma^{-}$]
These permutations satisfy $\sigma(n-i)=i, 1\leq i\leq m$. Again, there are $m!$ possibilities to define $\sigma(i), 1<i\leq m$. Only the  involution $[n\, n-1\, \dots 2\, 1]$ is in the intersection of these sets. 

\item[An enhanced $m$-nesting in $A_\sigma^{+}$]
If the $m$-nesting uses only $2m-1$ vertices, there is one left over. It must either be a lower transient vertex, or a loop since there is nothing left to connect to it. 
We count these by considering the different ways to construct it from a smaller permutation diagram. Suppose we have a 
permutation with an $m$-nesting on $2m-1$ vertices. By the first part, we know there are $(m-1)!$ of these.
We place it on $2m$ points, by first selecting our special vertex
$i$, and placing the permutation on the rest. There are $2m$ ways to
pick this special vertex. Finally, we create the new permutation $\sigma$ by
connecting the new vertex to the rest of the structure. We choose a
point $j$ to be the value $\sigma(i)$. We can choose $i$ and thus $i$
is a loop. Otherwise, $j$ must be before the loop in $\sigma'$. We then set $\sigma^{-1}(i)$
to be $\sigma'^{-1}(j)$. There are $m$ choices for $\sigma(i)$. 

\item[Over counting] we have counted twice the family of diagrams with two loops in
the center. There are $(m-1)!$ of these. 
\end{description}
Putting all of the pieces
together, and simplifying the expression we get the formula: \[\operatorname{N}_{2m}(m)= 2(m+1)!-(m-1)!-1.\]
\end{proof}

This proof suggests a direct involution on the permutation which switches a maximum nesting for a maximum crossing, since the degree sequence of a nesting and a crossing have the same shape. Thus the formula for maximum crossings is the same. This involution might send a $k$-crossing to a $k+1$-nesting, and so it can not be used for general $k$. 

\subsection{Other enumerative questions} 
From the formula, we see that a very smal proportion of permutations have maximum crossings, ($\leq 2\frac{n+1}{2}!/n!$)
or are non-crossing $\approx 4^nn^{-3/2}/n!$. What can be said of the nature of the distribution, or the average crossing number? What is the nature of the generating function $P(z;u)$ where $u$ marks the crossing number, or even simply the generating function for $k$-noncrossing permutations? Bousquet-M\'elou and Xin~\cite{BoXi05} consider this question for permutations: 2-noncrossing partitions are counted by Catalan numbers, (as we mentioned before), and thus the generating function is algebraic; the counting sequence for 3-noncrossing partitions is P-recursive, and so the generating function is D-finite, and they conjecture that the generating function for $k$-noncrossing partitions, $k>3$ are likely not D-finite. How can these results be adapted to permutations, given the similar structure?

\section{Proof of main theorem}
The proof of the main theorem first decomposes a permutation into its upper and lower arc diagrams and then applies the results for set partitions separately to each part. 

\begin{proof}
We can apply a result of deMier~\cite[Theorem 3.3]{deMi07} almost directly to prove our result by describing an involution $\Psi:\Sn\rightarrow\Sn$ which preserves the left-right degree sequence and swaps nesting and crossing number. That is, $D_\sigma=D_{\Psi(\sigma)}$, and $\Ne(\sigma)=\Cr(\Psi(\sigma))$, $\Cr(\sigma)=\Ne(\Psi(\sigma))$. This involution happens on the Ferrer's diagram, and is described in detail by Krattenthaler~\cite{Kratt06}.  

Her Theorem 3.3 states that any left-right degree sequence $D$, the number of $k$- 
noncrossing graphs on $D$ equals the number of $k$-nonnesting graphs on $D$. Actually,  she proves more: the number of graphs with crossing number $i$ and nesting number $j$ have a symmetric distribution across all graphs with a given fixed degree sequence. 

In order to apply her results, the first step is re-write $A_\sigma^+$ so that we only consider proper crossings and nestings instead of enhanced crossing and nestings. This is a common trick, known as inflation. Essentially, we create the graph $g$ from $A_\sigma^+$ by adding some supplementary vertices to eliminate loops and transitory vertices:\\
\begin{center}
\includegraphics[height=2cm]{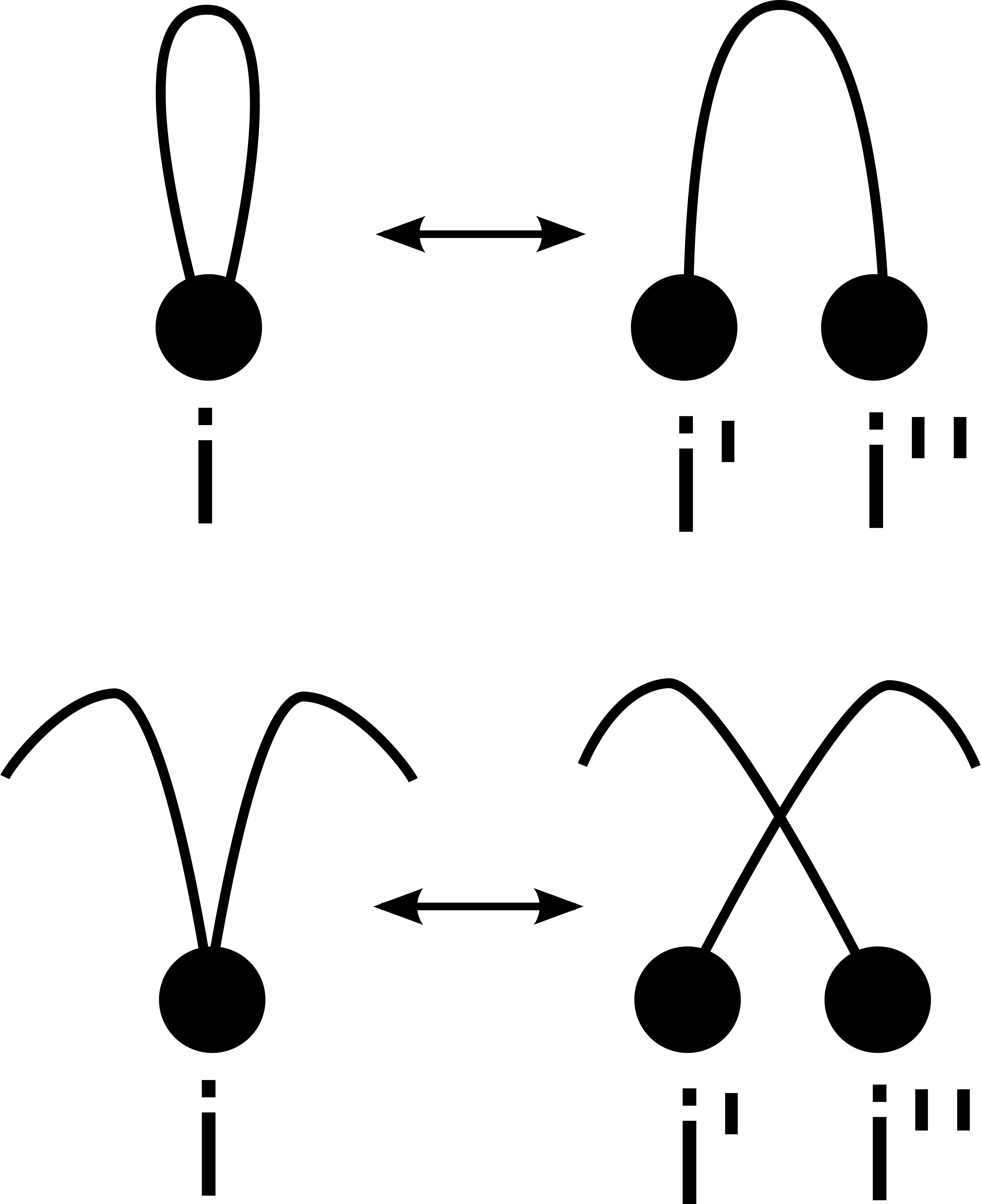}
\end{center}

Now each nesting and crossing is proper, and by~\cite[Lemma 3.4]{deMi07} $\Ne^*(A_\sigma^+)=\Ne(g)$ and $\Cr^*(A_\sigma^+)=\Cr(g)$. 

In the course of her proof, she relies on an involution of Krattenthaler~\cite{Kratt06} on fillings of Ferrer's diagrams. Let $\Psi$ be the map on embedded labelled graphs described implicitly in her proof. Because $\Psi$ is a left-right degree preserving map, we can identify the supplementary vertices in $\Psi(g)$ to get a graph with the correct kind of vertices. Call this new graph $g'$. We now extend the definition of $\Psi$ to $A_\sigma^+$ by $\Psi(A_\sigma^+)\equiv g'.$

Consider the pair of graphs $(\Psi(A_\sigma^+), \Psi(A_\sigma^{-}))$.

Proving our main theorem now reduces to showing that there is a unique $\tau\in\Sn$ such that $A_\tau=(\Psi(A_\sigma^+), \Psi(A_\sigma^{-}))$, which we do next.
For every vertex in $A_\tau$ the indegree and the outdegree are equal to one. This is because the left-right degree sequence of both the top and the bottom are preserved in the map, and hence the vector sum of their degree sequence is unchanged, i.e. $(1,1)^n$, and has all the correct partial sum properties.  The map is a bijection and so $\tau$ is unique. 

This map swaps the upper nesting and the upper crossing number, and also the lower nesting and the lower crossing number. Thus
$\Cr(\tau)=\max\{\Cr^*(A_\tau^+),\Cr(A_\tau^{-})\}=\max \{ \Ne^*(A_\sigma^+), \Ne(A_\sigma^{-}) \}=\Ne(\sigma)$. Thus, the crossing and the nesting number are switched under the map $\Psi$.  
\end{proof}

\subsection{Other proofs of this result} Set partitions can be viewed as an instance of fillings of Ferrer's diagrams, and Krattenthaler's involution reduces to the involutions that Chen \emph{ et al.} describe to prove their~\cite[Theorem 1]{Chetal07}.

Figure~\ref{fig:psisigma} illustrates our involution on an example. Remark that the degree sequence is fixed.
\begin{figure}[h!]\center 
\large{$A_\sigma=$} \includegraphics[width=6cm]{perm.png} \\ \bigskip
\large{$A_{\Psi(\sigma)}=$} \includegraphics[width=6cm]{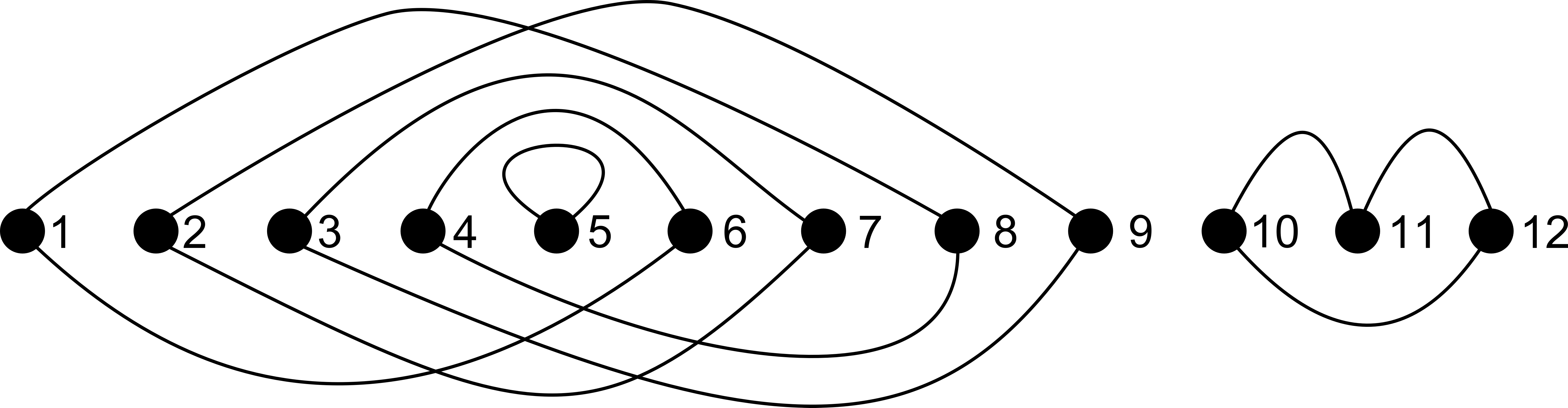} \\ \bigskip
\caption{\it The permutation $\sigma$ and its image in the involution $\Psi(\sigma)$. Note that $\Ne(\Psi(\sigma))=4$, $\Cr(\Psi(\sigma))=3$.}
\label{fig:psisigma}
\end{figure} 
 
\subsection{Equidistribution in permutation subclasses}
Involutions are in bijection with partial matchings, and have thus been previously considered. What of other subclasses of permutations? The map presented here does not fix involutions, because loops are mapped to upper transient vertices, but it does fix any class that is closed under degree sequence, for example, permutations with no lower transitory vertices, or permutations with no upper transitory vertices nor loops. These conditions have interpretations in terms of other permutation statistics, if we consider the initial motivations of Corteel. 

\section{Conclusions and open questions}
The main open question, aside from the enumerative, and probabilistic questions we have already raised, is to find a direct permutation description of our involution, i.e. a description avoiding the passage through tableaux or fillings of Ferrer's diagrams. Is this involution already part of the vast canon of permutation automorphisms?  

Which subclasses of permutations preserve the symmetric distribution? From our example, we remark that cycle type is not neccesarily conserved (since loops are always mapped to upper transitory vertices), but non-intersecting intervals are preserved. Involution permutations are in bijection with partial matchings, and so this subclass has this property. 

Is there an interpretation of  crossing and nesting numbers in terms of other permutations statistics? Which other statistics does this involution preserve? 

Ultimately we have considered a type of graph with two edge colours and strict degree restrictions. Can this be generalized to a larger class of graphs with fewer degree restrictions?
What of a generalization of graphs with multiple edge colours?


\subsection*{Acknowledgements} The authors thank Cedric Chauve, Lily Yen, Sylvie Corteel and Eric Fusy for useful discussions. SB and MM are both funded by NSERC (Canada), and MM did part of this work as part of a CNRS Poste Rouge (France).

\end{document}